\documentclass[11pt]{amsart}
\usepackage[shortalphabetic]{amsrefs}
\usepackage[utf8]{inputenc}
\usepackage{amsthm}
\usepackage[utf8]{inputenc}
\usepackage{appendix}
\usepackage{cite}
\usepackage{amssymb,amsmath,amsthm}
\usepackage{hyperref}

\newcommand{\amsprimary}[1]{{\footnotesize\noindent AMS 2010 \textit{Mathematics subject
classification:} Primary #1\vspace{1pc}}}
\newcommand{\keywordsnames}[1]{{\footnotesize\noindent\textit{Key words:} #1\vspace{1pc}}}

\newtheorem{theorem}{Theorem}

\newtheorem{teo}{Theorem}

\newtheorem{lemma}[teo]{Lemma}

\theoremstyle{definition}

\theoremstyle{remark}

\title[]{On Milnor's criterion for deciding whether a surface is hyperbolic or parabolic for biharmonic functions}
\author{John E. Bravo \and Jean C. Cortissoz }
\email{jcortiss@uniandes.edu.co}
\address{Department of Mathematics, Universidad de los Andes, Bogot\'a DC, Colombia}
\date{}

\begin{document}

\maketitle

\begin{abstract}
In this paper we generalise a celebrated result of Milnor that characterises whether a 
rotationally symmetric surface is
parabolic or hyperbolic to the case of biharmonic functions.
\end{abstract}

\keywordsnames{Liouville's Theorem; bounded harmonic functions}

{\amsprimary {31C05, 53C21}}

\tableofcontents

\section{Introduction}

In 1977 Milnor in \cite{Mil} published a criterion (which 
according to himself surprised him), based on the curvature to decide
whether $\mathbb{R}^2$ with a metric of the form
\begin{equation}
\label{eq:rotationally_symmetric}
g=dr^2+\phi\left(r\right)^2 d\theta^2.
\end{equation}
was either hyperbolic (it admits bounded nonconstant harmonic functions) or parabolic
(every bounded harmonic function must be constant).
Namely he showed the following.
\begin{theorem}
    Let $\mathbb{R}^2$ with a complete metric of the form (\ref{eq:rotationally_symmetric}).
    If the curvature is greater than or equal to $-1/r^2\log r$, then every bounded
    harmonic function must be constant. On the other hand, if there
    is an $\epsilon>0$ such that the curvature outside a compact
    subset of $\mathbb{R}^2$ is less than or equal to $-\left(1+\epsilon\right)/r^2\log r$,
    then there are non constant bounded harmonic functions.
\end{theorem}

Of course, we need to define what is understood by a harmonic function
in a surface with an arbitrary metric, and, incidentally, what we mean by
a metric. This we shall do in Section \ref{section:preliminaries}, but for now, 
let us think that we are in $\mathbb{R}^2$ with its classical
Euclidean structure, which in our context corresponds to
having $\phi\left(r\right)=r$. In this case, in cartesian 
(not polar coordinates), the Laplacian, which depends for its definition
on the metric given to $\mathbb{R}^2$, is given (as the reader most 
probably knows) by
\[
\Delta =\frac{\partial^2}{\partial x^2}+\frac{\partial^2}{\partial y^2},
\]
and a function is harmonic if $\Delta u=0$.
The case of $\mathbb{R}^2$ with the usual metric, corresponds to
curvature 0 (for the definition of curvature see Section \ref{section:preliminaries}), and Milnor's 
result generalizes Liouville's theorem: bounded harmonic functions
are constant.

\medskip
After Milnor's work, very little has been done regarding how the behaviour of biharmonic functions
is affected by curvature. Let us introduce the relevant jargon.
We say that a function $u$ is $k$-harmonic if 
\[
\left(-\Delta\right)^k u=0.
\]
For instance, when $k=1$, we recover the usual harmonic functions, and when
$k=2$, we have
\[
\Delta^2 u =0,
\]
that is, we have the biharmonic functions. In general, we call the
family of functions for $k\geq 2$ polyharmonic.

\medskip
It is well known that in the euclidean case
any bounded polyharmonic function is constant. There are several ways 
of proving this result, one of them, and perhaps the shortest, 
is using the theory of tempered 
distributions and the Fourier transform (see \cite{Huilgol,Weck} ).

\medskip
It is our purpose in this paper to extend Milnor's criteria to 
the family of biharmonic functions. Indeed, we shall first
show the following.
\begin{theorem}
\label{thm:main_1}
    Let $\mathbb{R}^2$ with a metric 
    \[
    g=dr^2+\phi\left(r\right)^2d\theta^2.
    \]
     with $\phi\left(r\right)$ so that the metric is smoothly extendable up to the origin
     (that is $\phi\left(0\right)=0$, $\phi'\left(0\right)=1$, and all the even derivatives
     of $\phi$ vanish at the origin).
    If outside of a compact subset of $\mathbb{R}^2$ the curvature satisfies $K\geq -\dfrac{1}{r^2\log r}$
    and $\phi\left(r\right)\rightarrow\infty$ as $r\rightarrow \infty$,
    then any bounded
    biharmonic function is constant.
\end{theorem}
The proof of this result is based on the method of separation of variables and Plancherel's
theorem. This theorem can be generalised to the case of polyharmonic functions, but in our attempt
to make our exposition clear, we have only included the proof in the case of biharmonic functions.

\medskip
By refining the methods employed in the proof of Theorem \ref{thm:main_1}, we arrive at
the following result.

\begin{theorem}
\label{thm:main_2}
 Let $\mathbb{R}^2$ with a metric 
    \[
    g=dr^2+\phi\left(r\right)^2d\theta^2,
    \]
    with $\phi\left(r\right)$ nondecreasing outside a compact subset of $\mathbb{R}^2$ 
    and so that the metric is extendable up to the origin.
    
    Assume that there is an $\epsilon>0$ such that outside a compact set the curvature
    $K$ satisfies $K\leq -r^{2+\epsilon}$, then
    there are bounded biharmonic functions that are not harmonic. On the other hand, if there are
    positive numbers $\eta>0$ and $\epsilon>0$ such that outside a compact subset of $\mathbb{R}^2$
    it holds that
    $-\eta r^2\leq K\leq -\frac{1+\epsilon}{r^2\log r}$, then any bounded
    biharmonic function must be harmonic.
\end{theorem}
This theorem is clearly a generalization of the second part of Milnor's theorem; however it contains a significant surprise:
there is region, namely when the curvature
is in between $-\eta r^{2}$ and $-\frac{1+\epsilon}{r^2\log r}$
for $\epsilon>0$, where bounded biharmonic functions must be harmonic, which
can be also be seen as a Liouville-type theorem.

\medskip
This paper originated from a question posed to the second author by Michael Ruzhansky during
a talk. As it this question was connected to Milnor's celebrated result, we decided to write
this paper to answer Ruzhansky's question -which was a bit more general, at least in the case of
a rotationally symmetric surface which is still
interesting and contains some surprises as well. Although this paper may look somewhat technical (at a first glance),
the computations are elementary as they only need Calculus (including a bit of differential equations) 
and elementary Fourier Analysis;
we have included a quick primer on basic Riemannian Geometry of rotationally symmetric surfaces
so that even a reader without much knowledge on this subject may follow our arguments and
even get interested in this beautiful area of Mathematics.

\section{Preliminaries}
\label{section:preliminaries}

To make this paper readable and an introduction
to the subject, we want to introduce the basic concepts of
Riemannian geometry needed to understand the results of this paper.
The reader may also consult \cite{Co}, where such an introduction
is also given. We shall work with polar coordinates $\left(r,\theta\right)$
in $\mathbb{R}^2$. Thus, given any vector $v_p$
anchored at a point $p\in \mathbb{R}^2$ different from the
origin, it can be decomposed in its radial and tangential
component to the circle of radius $r$ as
\[
v_p=a\frac{\partial}{\partial r}+b\frac{\partial}{\partial \theta},
\]
where $a$ and $b$ are the respective radial and tangent components,
and at $\left(x,y\right)$ in cartesian coordinates, we have
\[
\frac{\partial}{\partial r}=x\mathbf{i}+y\mathbf{j}, \quad 
\frac{\partial}{\partial\theta}=-y\mathbf{i}+x\mathbf{j},
\]
were $\mathbf{i}$ and $\mathbf{j}$ represent the unit coordinate vectors
in the direction of the coordinate axes.

\medskip
We say that the $g$ metric is given
at $p=\left(r,\theta\right)$ by
\[
g=dr^2+\phi\left(r\right)^2d\theta^2,
\]
if at that point we make
the assignment
\[
g\left(v_p,v_p\right)=
\left\|v_p\right\|^2=a^2+b^2\phi\left(r\right)^2.
\]
If $\phi\left(0\right)=1$ and all its even derivatives $\phi^{\left(2k\right)}\left(0\right)=0$,
then $g$ can be smoothly extended up to the origin. As familiar cases, when 
$\phi\left(r\right)=r$ we obtain the usual euclidean metric, and when
$\phi\left(r\right)=\sinh r$ we have the hyperbolic plane.

The curvature for a metric $g$ can be computed and turns out to be
\[
K=-\frac{\phi''}{\phi}.
\]
Given a smooth metric $g$ defined in $\mathbb{R}^2$, the Laplacian is defined by
\begin{equation}
\label{eq:laplacian}
\Delta = \frac{\partial^2}{\partial r^2} +\frac{\phi'}{\phi}\frac{\partial}{\partial r}
+\frac{1}{\phi^2}\frac{\partial^2}{\partial \theta^2}.
\end{equation}
As the metric can be smoothly extended to the origin, so can be the Laplacian.
It is also interesting to point out that when the metric is smooth, it can be shown
that any polyharmonic function in $\mathbb{R}^2$ is smooth (see Chapter 3 in \cite{Schechter})

\medskip
Also, we will need the following Comparison Theorem due to Sturm (see \cite{Co}).
\begin{lemma}[Comparison Lemma]
Let $h, f: \left[a,\infty\right)\longrightarrow \left(0,\infty\right)$ be $C^2$ functions. If
$\displaystyle \frac{f''}{f}\left(r\right)\leq \frac{h''}{h}\left(r\right)$, $r>a$, and
$\displaystyle\frac{f'}{f}\left(a\right)\leq\frac{h'}{h}\left(a\right)$ then 
$\displaystyle\frac{f'}{f}\left(r\right)\leq\frac{h'}{h}\left(r\right)$ for $r>a$.
\end{lemma}

Let us examine some important consequences of the Comparison Lemma that will be useful later. 
For instance, if $K\geq -\dfrac{1}{r^2\log r}$ outside of a compact
subset of $\mathbb{R}^2$ and $\phi$ is nondecreasing, then 
\[
\lim_{r\rightarrow \infty}\frac{\phi'\left(r\right)}{\phi\left(r\right)}=0.
\]
Indeed, the inequality assumed for the curvature implies the inequality
\[
\frac{\phi''}{\phi}\leq \frac{1}{r^2\log r},
\]
for $r> R_0\geq 1$.
Let $\psi=\left(r-R_0+1\right)\log \left(r-R_0+1\right)$ for $r>R_0\geq 1$. Then we can compute
\[
\frac{\psi'\left(r\right)}{\psi\left(r\right)}=\frac{\log \left(r-R_0+1\right)+1}{\left(r-R_0+1\right)\log\left(r-R_0+1\right)},
\]
and 
\[
\frac{\psi''\left(r\right)}{\psi\left(r\right)}=\frac{1}{\left(r-R_0+1\right)^2\log \left(r-R_0+1\right)}.
\]
Observe then that for $\eta>0$ small, 
\[
\frac{\phi'\left(R_0+\eta\right)}{\phi\left(R_0\right)}\leq \frac{\psi'\left(R_0+\eta\right)}{\psi\left(R_0+\eta\right)}
\quad \mbox{and} \quad
\frac{\phi''\left(r\right)}{\phi\left(r\right)}\leq \frac{\psi''\left(r\right)}{\psi\left(r\right)}.
\]
By the Comparison Lemma then for $r>R_0$
\[
0\leq \frac{\phi'\left(r\right)}{\phi\left(r\right)}\leq \frac{\psi'\left(r\right)}{\psi\left(r\right)},
\]
and our claim follows by the Squeeze Theorem.

\subsection{A uniqueness theorem for biharmonic functions} 
Here we need to introduce a Maximum Principle for a special elliptic system,
which of course could be made more general, but which is more
than enough for our purposes.  The main application 
is a uniqueness theorem for solutions to systems that we shall need below.

Consider the following system
\[
\left\{
\begin{array}{l}
\Delta u =0 \quad \mbox{in} \quad \Omega,\\
\Delta v = u \quad \mbox{in} \quad \Omega,
\end{array}
\right.
\]
where $u,v\in C^{2}\left(\Omega\right)\cap C\left(\overline{\Omega}\right)$.
Notice that $v$ is biharmonic and that any biharmonic function can be written this way.
Consider the expresion
\[
w=\frac{1}{2}\left(u^2+\alpha v^2\right).
\]
A calculation shows that
\[
\Delta w = \alpha uv +\left|\nabla u\right|^2+\alpha\left|\nabla v\right|^2.
\]
We can estimate as follows. Use the Peter-Paul inequality 
\[
uv\leq \frac{1}{2\sqrt{\alpha}}u^2+\frac{\sqrt{\alpha}}{2}v^2=\frac{1}{\sqrt{\alpha}}w,
\]
and replace in the expression above to get
\[
\Delta w \geq -\sqrt{\alpha}w+\left|\nabla u\right|^2+\alpha\left|\nabla v\right|^2
\geq -\sqrt{\alpha} w.
\]

Let $f_1$ the first Dirichlet eigenfunction for the Laplacian on a domain $\Omega'\supset \Omega$.
Then, as is well known, we can take $f_1>0$.
A calculation shows that 
\[
\Delta\left(\frac{w}{f_1}\right)=\frac{\Delta w}{f_1}-2\nabla\left(\frac{w}{f_1}\right)
+\lambda_1^2\frac{w}{f_1}+4\left(\frac{\left|\nabla f_1\right|}{f_1}\right)^2\frac{w}{f_1},
\]
and we obtain the differential inequality
\[
\Delta\left(\frac{w}{f_1}\right)+\nabla\left(\frac{w}{f_1}\right)\geq \left(-\sqrt{\alpha}
+\lambda_1^2\right)\frac{w}{f_1}.
\]
Thus, if $\alpha>0$ is small enough then $-\sqrt{\alpha}
+\lambda_1^2>0$, and hence, by the maximum principle (Theorem 6, Chapter 2, p. 64 in \cite{Protter}) if 
$w/f_1=0$ on the boundary of $\Omega$, then $w/f_1=0$ in $\Omega$, and hence 
$w=0$ in $\Omega$.

\medskip
From the previous discussion follows the following. 
\begin{lemma}
Let $u$ and $v$ biharmonic functions defined 
in $\Omega$ and $C^2$ in an open neighborhood of $\overline{\Omega}$. If $u=v$ and $\Delta u=\Delta v$
in $\partial \Omega$, then $u=v$ in $\Omega$.
\end{lemma}
\section{A proof of Theorems \ref{thm:main_1} and \ref{thm:main_2}}

\subsection{Theorem \ref{thm:main_1}}
In \cite{Co2} it is shown that to solve the equation
\begin{equation}
\label{eq:Dirichlet}
\Delta u = 0 \quad\mbox{on}\quad M.
\end{equation}
Let us recall the procedure.
We use separation of variables and for an integer $m$ we write a solution 
to (\ref{eq:Dirichlet}) as
\[
u= \varphi_m\left(r\right)e^{im\theta}.
\]
This consideration, using (\ref{eq:laplacian}) leads to an ordinary differential equation 
for $\varphi_m$, namely
\begin{equation}
\label{eq:homogeneous}
\varphi''_m\left(r\right)+\frac{\phi'}{\phi}\varphi_m'\left(r\right)-\frac{m^2}{\phi^2}\varphi_m\left(r\right)
=0.
\end{equation}
It is easy to show that a (non-singular) solution is given in this case by
\[
\varphi_m\left(r\right)=\exp\left(\int_1^r \frac{\left|m\right|}{\phi\left(s\right)}\,ds\right).
\]
From this we arrive at the following fact. Any harmonic function
on $M$ can be written as
\[
\sum_{m\in \mathbb{Z} }\varphi_m\left(r\right)e^{im\theta}.
\]
We want to apply the same method to write any biharmonic function defined 
on $M$ as a Fourier series. That is, we want to find all possible solutions to the equation
\[
\Delta^2 u =0 \quad \mbox{in} \quad \mathbb{R}^2,
\]
and to write them explicitly as a Fourier series.
To attempt a solution, we shall write the previous equation as a system. That is 
we want to find a solution to the system
\begin{equation}
\label{eq:system}
\left\{
\begin{array}{l}
\Delta v = 0 \quad \mbox{in}\quad \mathbb{R}^2,\\
\Delta u = v \quad \mbox{in}\quad \mathbb{R}^2.
\end{array}
\right.
\end{equation}
Notice that if $u$ is part of a solution to the system above, then $u$ must be biharmonic.
If we write our attempt for a solution as $v=\psi_m\left(r\right)e^{im\theta}$ for an integer $m$,
with $u=\varphi_m\left(r\right)e^{im\theta}$, which is known to be harmonic,
then
$\psi_m\left(r\right)$ must satisfy the equation
\begin{equation}
\label{eq:nonhomogeneous}
\psi''_m\left(r\right)+\frac{\phi'}{\phi}\psi_m'\left(r\right)-\frac{m^2}{\phi^2}\psi_m\left(r\right)
=\varphi_m\left(r\right).
\end{equation}
Since $\varphi_m$ is a solution to the homogeneous equation (that is, when the
righthandside of (\ref{eq:nonhomogeneous}) is 0), we can use
reduction of order to compute $\psi_m$. So we look for a solution
to (\ref{eq:nonhomogeneous}) of the form $z\left(r\right)\varphi_m\left(r\right)$. Then 
$z$ satisfies an equation
\begin{equation}
z''+\left(\frac{2\left|m\right|}{\phi}+\frac{\phi'}{\phi}\right)z'=1,
\end{equation}
where we have used the fact that
\[
\varphi_m'= \frac{\left|m\right|}{\phi}\varphi_m.
\]
Solving for $z$, we obtain 
\[
z\left(r\right)=
\int_0^r \frac{1}{\phi\left(s\right)\varphi_{2m}\left(s\right)}\int_0^s\phi\left(t\right)\varphi_{2m}\left(t\right)\,dt\,ds.
\]
Therefore the pair $v=\varphi_m\left(r\right)$, $u=\psi_m\left(r\right)=z\left(r\right)\varphi_m\left(r\right)$ 
solves the system (\ref{eq:system}).

\medskip
Next, we show that any biharmonic function $u:M\longrightarrow \mathbb{R}$ can be written as
a Fourier series
\begin{equation}
\label{eq:Fourier_expansion}
    \sum_{m\in \mathbb{Z}} \left(c_m\varphi_m\left(r\right)+d_m\psi_m\left(r\right)\right)e^{im\theta}.
\end{equation}
To do so, given a biharmonic $u$ we can solve the following boundary value problem on $\partial B_R$.
\[
\left\{
\begin{array}{l}
\Delta^2 v = 0 \quad \mbox{in} \quad B_R,\\
v = u|_{\partial B_R} \quad \mbox{on}\quad \partial B_R,\\
\Delta v=\Delta u|_{\partial B_R}\quad\mbox{on}\quad \partial B_R.
\end{array}
\right.
\]
The function $v$ can be represented using a Fourier expansion. Indeed, to satisfy the
boundary conditions we need to solve for the coefficients $c_m$ and $d_m$. 
Observe that if $u_R$ can be represented as
\[
u_R=\sum \alpha_me^{im\theta},
\]
and $\Delta u_R$ can be represented as
\[
\Delta u_R= \sum \beta_me^{im\theta}.
\]
This 
leads us to the following system of equations.
\[
\left\{
\begin{array}{r}
c_m\varphi_m\left(R\right)+d_m\psi_m\left(R\right)=\alpha_m\\
d_m\varphi_m\left(R\right)=\beta_m,
\end{array}
\right.
\]
a system which can clearly be solved.
By uniqueness, $u=v$ on $B_R$, and hence, when restricted to $B_R$, $u$
is represented by the Fourier series above.

If we consider a ball $B_{R'}\supset B_R$. Using the procedure employed above, we have a representation
for $u$ on $B_{R'}$ as a Fourier series
\begin{equation}
\label{eq:Fourier_expansion_1}
    \sum_{m\in \mathbb{Z}} \left(c'_m\varphi_m\left(r\right)+d'_m\psi_m\left(r\right)\right)e^{im\theta}.
\end{equation}
But this would also be a representation for $u$ on $B_R$, and hence, by uniqueness,
\[
c'_m= c_m,\quad d_m'=d_m.
\]
Our claim immediately follows.

\medskip
The following lemma establishes, under a condition on the curvature, a relation between 
the $\psi_m$ and its respective $\varphi_m$, namely, that $\psi_m$ is much larger than
$\varphi_m$.
\begin{lemma}
\label{lemma:growth_z}
    If $K\geq -\dfrac{1}{r^2\log r}$ outside of a compact subset
    of $\mathbb{R}^2$, and $\phi\rightarrow \infty$, then $z\left(r\right)\rightarrow\infty$ as $r\rightarrow\infty$.
\end{lemma}
\begin{proof}
    We actually show that, under the hypothesis, 
    \[
    \frac{1}{\phi\left(s\right)\varphi_{2m}\left(s\right)}\int_0^s\phi\left(t\right)\varphi_{2m}\left(t\right)\,dt\rightarrow
    \infty
    \]
    as $s\rightarrow \infty$. To do this, we use L'Hôpital's rule.
    \begin{eqnarray*}
     \lim_{s\rightarrow \infty}\frac{1}{\phi\left(s\right)\varphi_{2m}\left(s\right)}
     \int_0^s\phi\left(t\right)\varphi_{2m}\left(t\right)\,dt&=&
    \lim_{s\rightarrow \infty}\frac{\phi\left(s\right)\varphi_{2m}\left(s\right)}{\phi'\left(s\right)
    \varphi_{2m}\left(s\right)+2m\varphi_{2m}\left(s\right)}\\
    &=&
    \lim_{s\rightarrow \infty}\frac{1}{\frac{\phi'\left(s\right)}{\phi\left(s\right)}
    +\frac{2m}{\phi}},\\
    \end{eqnarray*}
    and using that the hypothesis on the curvature implies that $\phi'/\phi \rightarrow 0$, the result follows.
\end{proof}

\medskip
Once we have that any biharmonic function
$u$ can be represented as a Fourier expansion
(\ref{eq:Fourier_expansion}), a proof of Theorem \ref{thm:main_1} can be given as follows.

\medskip
Using Plancherel's theorem, we can compute
\begin{eqnarray*}
    \left\|u\right\|_{L^2\left(\partial B_R\right)}^2&=& \sum_m \int_0^{2\pi}
    \left|c_m\varphi_m\left(r\right)+d_m\psi_m\left(r\right)\right|^2\phi\left(r\right)\,dr.
\end{eqnarray*}
If $u$ is bounded, then
\[
 \left\|u\right\|_{L^2\left(\partial B_R\right)}^2 = O\left(1\right)\phi\left(r\right).
\]
Under the assumption that $K\geq -\dfrac{1}{r^2\log r}$, by Lemma \ref{lemma:growth_z}
(since $\psi_m= z\varphi_m$)
\[
\varphi_m\left(r\right)=o\left(\psi_m\left(r\right)\right),
\]
from which we get that, if $m$ is such that $d_m\neq 0$, then 
\[
 \left\|u\right\|_{L^2\left(\partial B_R\right)}^2\geq c\psi_m\left(r\right)^2\phi\left(r\right)
\]
and since $\psi_m\left(r\right)\rightarrow \infty$ as $r\rightarrow \infty$, we reach
a contradiction. Thus, we must have that $d_m=0$ for all $m\in \mathbb{Z}$. On the other hand, it
must also hold that $c_m=0$ for all $m\in \mathbb{Z}$, because as it is shown in 
\cite{Co3}, when 
$K\geq -1/r^2\log r$, it also holds that $\varphi_m\left(r\right)\rightarrow\infty$ as $r\rightarrow\infty$.
This finishes the proof of Theorem \ref{thm:main_1}.

\subsection{Theorem \ref{thm:main_2}}
Next we present a proof of the criterion \` a la Milnor. Here the estimates are much more delicate.
We begin with a somewhat technical but useful lemma that relates a bound on the curvature
with the behavior of the function $\phi$.
\begin{lemma}
\label{lemma:inequalities_1}
\begin{itemize}
\item If $K\left(r\right)\leq -r^{2+\epsilon}$ then there exists a constant $B>0$ such that
    \[
    \frac{1}{\phi\left(s\right)}\int_0^s \phi\left(\sigma\right)\,d\sigma\leq 
    e^{-Bs^{2+\epsilon}}\int_0^r e^{Bt^{2+\epsilon}}\,dt.
    \]
\item If $K\left(r\right)\geq -Mr^{2}$, $M>0$, then there exists a constant $B>0$ such that
    \[
    \frac{1}{\phi\left(s\right)}\int_0^s \phi\left(\sigma\right)\,d\sigma\geq 
    e^{-B s^{2}}\int_0^r e^{Bt^{2}}\,dt.
    \]
\end{itemize}    
\end{lemma}
\begin{proof}
Let
\[
\psi\left(s\right)=\frac{1}{\phi\left(s\right)}\int_0^s \phi\left(\sigma\right)\,d\sigma.
\]
Then we can compute
\[
\psi'\left(s\right)=1-\frac{\phi'\left(s\right)}{\phi\left(s\right)}\psi\left(s\right).
\]
Using the Comparison Lemma we have that $\phi'/\phi\geq a r^{1+\epsilon}$.
Indeed,
consider the function $\zeta\left(r\right)=e^{\delta r^{2+\epsilon}}$.
A calculation shows that
\[
\frac{\zeta'}{\zeta}=\delta\left(2+\epsilon\right) r^{1+\epsilon},
\]
and
\[
\frac{\zeta''}{\zeta}=\delta^2\left(2+\epsilon\right)r^{2+2\epsilon}
+\delta\left(2+\epsilon\right)\left(1+\epsilon\right)r^{\epsilon}.
\]
Let $R_0$ be such that if $r\geq R_0$, $K\leq -r^{2+\epsilon}$. Pick $\delta>0$ small enough
so that
\[
\frac{\phi'}{\phi}\left(R_0\right) \geq \frac{\zeta'}{\zeta}\left(R_0\right),
\]
and that
\[
\frac{\phi''}{\phi}\left(r\right)\geq r^{2+\epsilon}\geq \frac{\zeta''}{\zeta}\left(r\right),
\quad \mbox{on}\quad r\geq R_0.
\]
From the Comparison Lemma follows our claim.
 We thus get
the differential inequality
\[
\psi'\left(s\right)+ar^{1+\epsilon}\psi\left(s\right)\leq 1.
\]
From this differential inequality follows that
\[
\psi\left(r\right)\leq \frac{1}{e^{\frac{a}{2+\epsilon}}r^{2+\epsilon}}
\int_{R_0}^r e^{\frac{a}{2+\epsilon}s^{2+\epsilon}}\, ds +
\frac{e^{\frac{a}{2+\epsilon}R_0^{2+\epsilon}}}{e^{\frac{a}{2+\epsilon}r^{2+\epsilon}}}
\psi\left(R_0\right).
\]

\medskip
In the second case $\phi'/\phi \leq r$. The proof is similar 
to the previous one, but this time we take $\zeta\left(r\right)=e^{Ar^2}$.
In this case
\[
\frac{\zeta'}{\zeta}=2Ar, \quad \frac{\zeta''}{\zeta}=4A^2r^2+2A
\]
If $R_0$ is such that when $r\geq R_0$, $K\geq -Mr^2$, then by taking
$A>0$ large enough we may assume that 
\[
\frac{\phi'}{\phi}\left(R_0\right)\leq \frac{\zeta'}{\zeta}\left(R_0\right),
\]
and that on $r\geq R_0$
\[
\frac{\phi''}{\phi}\left(r\right)\leq Mr^2 \leq \frac{\zeta''}{\zeta}\left(r\right).
\]
Using the Comparison lemma, we obtain that on $r\geq R_0$ the inequality
$\phi'/\phi \leq 2Ar$. From this it follows that $\psi$ satisfies 
the differential inequality
\[
\psi'+2Ar \psi\geq 1,
\]
from which the inequality
\[
\psi\left(r\right)\geq \frac{1}{e^{Ar^2}}\int_{R_0}^r e^{As^2}\,ds+\frac{e^{AR_0^2}}{e^{Ar^2}}\psi\left(R_0\right).
\]

\end{proof}
From the previous theorem we find that, for $r>0$ large enough, there are constants $c,C>0$ 
such that for $\epsilon\geq 0$,
\[
\frac{c}{s^{1+\epsilon}}\leq
e^{-A{s^{2+\epsilon}}}\int_0^r e^{A{t^{2+\epsilon}}}\,dt\leq \frac{C}{s^{1+\epsilon}}.
\]
To prove this inequality, we recur to L'H\^opital's rule. Indeed, we can compute as follows
\[
\lim_{s\rightarrow\infty}\frac{\int_0^r e^{At^{2+\epsilon}}\,dt}{\frac{e^{As^{2+\epsilon}}}{s^{1+\epsilon}}}
=\lim_{s\rightarrow\infty}\frac{e^{A{s^{2+\epsilon}}}}{\left(2+\epsilon\right)Ae^{A{s^{2+\epsilon}}}-\frac{\left(1+\epsilon\right)e^{s^{2+\epsilon}}}{s^{2+\epsilon}}}
=\frac{1}{\left(2+\epsilon\right)A}.
\]
From this computation, our assertion follows. Using the previous calculation and Lemma 
\ref{lemma:inequalities_1},
we get the following.
\begin{theorem}
\label{thm:final_estimate}
    \begin{itemize}
\item If $K\left(r\right)\leq -r^{2+\epsilon}$ with $\epsilon>0$ then
    \[
    \frac{1}{\phi\left(s\right)}\int_0^s \phi\left(\sigma\right)\,d\sigma\leq 
    \frac{C}{s^{1+\epsilon}},
    \]
    and hence $z\left(r\right)$ remains uniformly bounded.
\item If $K\left(r\right)\geq -\eta r^{2}$, $\eta>0$, then
    \[
    \frac{1}{\phi\left(s\right)}\int_0^s \phi\left(\sigma\right)\,d\sigma\geq 
    \frac{c}{s},
    \]
    and hence $z\left(r\right)\rightarrow\infty$ as $r\rightarrow\infty$.
\end{itemize}   
\end{theorem}
Theorem \ref{thm:final_estimate} is enough to give a proof of Theorem \ref{thm:main_2}. 
Indeed, Theorem \ref{thm:final_estimate} shows that
if the curvature is $\leq -r^{2+\epsilon}$ outside a compact subset of $\mathbb{R}^2$ then
$\psi_m\leq C\varphi_m$, and hence, both remain bounded (here we use the fact
that if for $\epsilon>0$, outside a compact subset of $\mathbb{R}^2$, $K\leq -\left(1+\epsilon\right)/r^2\log r$,
and $\phi$ is eventually nondecreasing,
then the $\varphi_m$'s are bounded \cite{Co2}). Thus,
there are biharmonic functions that are not harmonic, namely
the functions $\psi_m\left(r\right)e^{im\theta}$ with $m\in\mathbb{Z}$. On
the other hand, if the curvature is larger or equal than $-\eta r^2$ outside
a compact subset of $\mathbb{R}^2$,
then the $\psi_m$'s are unbounded, and as we have shown 
before, then if $u$ is biharmonic and bounded, then 
the coefficients $d_m$ in its representation must
be zero, and it immediately follows that it must
be harmonic.

\end{document}